\documentclass[oneside, dvipsnames, 11pt]{amsart}

\usepackage[T2A,T1]{fontenc}
\usepackage[utf8]{inputenc}
\usepackage[english]{babel}

\usepackage{xcolor}
\usepackage{hyperref}

\definecolor{myblue}{RGB}{60, 80, 197}
\hypersetup{%
  colorlinks=true,
  linkcolor=OrangeRed,
  citecolor = myblue,
  urlcolor = Blue
}

\usepackage{tabularx}
\usepackage{array}      
\usepackage{ragged2e}

\usepackage{mathtext}
\usepackage{amsmath,amsthm,mathtools,amssymb}

\theoremstyle{definition}
\newtheorem{definition}{Definition}[section]

\newtheorem{corollary}[definition]{Corollary}
\newtheorem{theorem}[definition]{Theorem}

\newtheorem{lemma}[definition]{Lemma}
\theoremstyle{remark}

\makeatletter
\newtheorem*{rep@theorem}{\rep@title}
\newcommand{\newreptheorem}[2]{%
\newenvironment{rep#1}[1]{%
 \def\rep@title{{\bf #2 \ref{##1}}}%
 \begin{rep@theorem}}%
 {\end{rep@theorem}}}
\makeatother
\newreptheorem{theorem}{Theorem}

\usepackage[left=30mm, right=30mm, bottom=25mm, top=25mm,includefoot]{geometry}
\newcommand{\eqdef}{\vcentcolon=}

\DeclareMathOperator{\interior}{int}
\DeclareMathOperator{\maxx}{Max}
\DeclareMathOperator{\vertex}{Vert}
\DeclareMathOperator{\nexxt}{Next}
\DeclareMathOperator{\previous}{Prev}
\DeclareMathOperator{\initial}{in}
\DeclareMathOperator{\terminal}{ter}
\DeclareMathOperator{\conv}{conv}
\DeclareMathOperator{\lk}{lk}
\DeclareMathOperator{\Image}{Im}

\newcommand{\Z}{\mathbb{Z}}

\newcommand{\R}{\mathbb{R}}

\usepackage{epigraph}
\usepackage{textcase}
\usepackage{version}
\excludeversion{confidential}

\usepackage{caption}
\usepackage[colorinlistoftodos]{todonotes}

\usepackage{tikz}
\usepackage{tikz-cd} 
\usetikzlibrary{cd}
\usepackage{pgfplots}
\pgfplotsset{width=6cm,compat=newest}

\usepackage{pdfpages}

\begin{document}

  \title[On essential simplicial maps $S^3 \rightarrow S^2$]{On essential simplicial maps $S^3 \rightarrow S^2$}

  \author{Mikhail Bludov, Sergei Vad. Fomin, Oleg R. Musin}

  \date{\today}

  \address{Mikhail V. Bludov: Moscow Institute of Physics and Technology, 9 Institutskiy per., Dolgoprudny, Moscow Region, 141701, Russia}

  \email{bludov.mv@phystech.edu}

  \address{Sergei Vad. Fomin: Saint Petersburg Leonhard Euler International Mathematical Institute at PDMI RAS, 10 Pesochnaya nab., St. Petersburg, 197022, Russia}

  \email{sergei.vadimovich.fomin@gmail.com}

  \address{O. R. Musin: University of Texas Rio Grande Valley, School of Mathematical and Statistical
Sciences, One West University Boulevard, Brownsville, TX, 78520, USA}

  \email{oleg.musin@utrgv.edu}

  \begin{abstract}
    A fiber-uniform bound on the complexity of an essential simplicial map $S^3\rightarrow S^2$ is proven, and the tightness of the bound is investigated. It follows that the triangulation of the Hopf map constructed by Madahar and Sarkaria in \cite{madahar_sarkaria_simplicial_2000} is minimal in its homotopy class in terms of the number of 3-simplices in the triangulation of $S^3$.
  \end{abstract}

  \maketitle

  \section{Introduction}

  It goes back to S. Eilenberg that a continuous map $f\colon S^{4n-1}\rightarrow S^{2n}$ with non-zero Hopf invariant $H(f)\neq 0$ \textit{has complicated fibers}. More concretely, in \cite[Theorem II]{eilenberg_continuous_1940} the following theorem is proven\footnote{Beware that further in this paper the author provides a fallacious assertion, which contradicts Hopf invariant one. See \cite{may_appreciation_2012} for context. On the other hand, the correctness of this particular statement in the case $m=1$ can easily be checked. It also follows from the main theorem of this paper.}\textsuperscript{,}\footnote{A reincarnation of this statement is known as the exactness of a term in the EHP sequence, or as the interpretation of Hopf invariant as the obstruction to desuspension of a homotopy class in $\pi_{4n-1}(S^{2n})$. See e.g. \cite{putman_homotopy_nodate}.}:

  \begin{theorem}
        Let $f\colon K_1 \rightarrow K_2$ be a simplicial map, where $|K_1|\cong M$ is a triangulation of an oriented closed ($2m+1$)-manifold and $|K_2|\cong S^{m+1}$ ($m>0$) is a triangulation of the $(m+1)$-dimensional sphere $S^{m+1}$. The following properties are equivalent:
        \begin{enumerate}
        %\footnote{For the definition of Hopf number in this situation, refer the Section \ref{sec:preliminaries}.}
            \item $H(f)=0$;
            \item there exists a map $g$ homotopic to $f$ such that for some $p\in S^{m+1}$ the preimage $g^{-1}(p)$ is an ($m-1$)-dimensional polyhedron.
        \end{enumerate}
   \end{theorem}

  In \cite{madahar_sarkaria_simplicial_2000}, K. V. Madahar and K. S. Sarkaria constructed a triangulation of the Hopf fibration. It is a simplicial map $\eta\colon  S^3_{12}\rightarrow S^{2}_4$ between a $12$-vertex triangulation of the $3$-sphere and the $4$-vertex triangulation of the $2$-sphere. The map was later\footnote{In the original paper \cite{madahar_sarkaria_simplicial_2000}, a fallacious argument for minimality is provided.} \cite{madahar_simplicial_2002} observed by Madahar to be minimal in terms of the number of vertices of the $3$-sphere among the simpicial maps homotopic to the Hopf fibration. This is a consequence of the Eilenberg's theorem for $m=1$, as it is easy to deduce from it that the preimage of any vertex of the triangulation $S^2$ must contain at least $3$ vertices of the triangulation of $S^3$.

  \begin{definition}
        Let $K_1, K_2$ be pure simplicial complexes of dimensions $d_1$ and $d_2$ respectively. Let $\maxx(K_i)$ be sets of maximal faces of the complexes. For $s\in \maxx(K_2)$ and a simplicial map $f\colon K_1\rightarrow K_2$,  define
        \[\mu(f,s) \eqdef |\{\sigma \in \maxx(K_1)|f(\sigma)=s\}|.\]
  \end{definition}

  Equivalently, if you take an internal point $p$ of $|s|$, $\mu(f,s)$ is the number of maximal simplices of $K$ intersecting $f^{-1}(p)$. 
  
  Originally, $\mu$ was defined in \cite{musin_homotopy_2025} in the case of simplicial maps between triangulated spheres with the goal of generalizing the Sperner's lemma.
    
  The following theorem is the main result of this paper. 
    
  \begin{reptheorem}{mainthm}
    Let $K_1, K_2$ be simplicial complexes such that $|K_1|\cong M$ for an oriented closed $3$-manifold $M$ and $|K_2|\cong S^2$. Let $s$ be a triangle of $K_2$ and $f\colon K_1\rightarrow K_2$ be a simplicial map. If $H(f)\neq 0$, then $\mu(f,s)\geq 9$.
    
    For each $d\in \Z\setminus\{0\}$ a simplicial map $f\colon S^3\rightarrow S^2$ is constructed, such that $H(f)=d$ and $\mu(f,s)= 9$.\footnote{This claim may look counterintuitive, because it follows that the same 9-tetrahedral triangulated solid torus may be embedded in 
    $S^3$ (or $\R^3$), such that two "horizontal" circles in the solid torus have an arbitrary linking number. The explanation is that the triangulation of the torus does not have to be rectiliniear in any possible kind of sense, so there is no upper bound on the linking number in terms of the number of tetrahedra in the triangulation of the solid torus. One can imagine a three component $(3k,3)$-torus link being a part of the $1$-skeleton of a 9-tetrahedral triangulation of the solid torus, which realizes the linking number $k$.}
  \end{reptheorem}

   The main theorem can be interpreted as a statement on the complexity of the preimage $f^{-1}(p)$ of an internal point $p$ of a triangle $|s|$ in $S^2$, which is analogous to Eilenberg's theorem.

   It turns out that the map $\eta\colon  S^3_{12}\rightarrow S^{2}_4$ constructed by Madahar and Sarkaria has $\mu(\eta, s)=9$ for all $s\in \maxx(S^{2}_4)$, and $S^3_{12}$ consists of $9\cdot 4 = 36$ tetrahedra.

  \begin{corollary}
        The map $\eta\colon  S^3_{12}\rightarrow S^{2}_4$ is minimal among simplicial maps homotopic to the Hopf fibration not only in terms of the number of vertices in the triangulation of $S^3$ but also in terms of the number of tetrahedra.
  \end{corollary}

   The examples for which the bound is sharp are slight modifications of the maps constructed by Madahar in \cite{madahar_simplicial_2002}.

   It is worth noting that the minimality of Madahar and Sarkaria's map $\eta$ in terms of the number of tetrahedra \textit{in the class of (semi)-simplicial triangulations of circle bundles} was already observed in \cite{mnev_minimal_2020}, \cite{panina_minimal_2024}.

%  \section{Preliminaries} \label{sec:preliminaries} What is homology notation support linking hopf invariant sign positive orientation triangulation connection of hopf and homotopy group we denote geometric realization as |$\cdot$| unbased support

  \section{Main theorem}

  \subsection{The lower bound}

  Consider simplicial complexes $K_1$ and $K_2$ such that $|K_1|\cong M$ for some oriented closed $3$-dimensional manifold $M$ and $|K_2|\cong S^2$. Let $s=ABC$ be a $2$-simplex in $|K_2|$. Let $f\colon K_1 \rightarrow K_2$ be a simplicial map between complexes $K_1$ and $K_2$.Fix $p\in \interior(s)$. The $1$-cycle $f^{-1}(p)$ is a union of oriented circles $C_1,\dots, C_k$. By a \textbf{segment} $c_i(\sigma)$ of $C_i$ we mean its intersection with a $3$-simplex $|\sigma|$, which is a segment in $|\sigma|$ (if it is not empty).
%Orient $K$ and $K_0$, so that $s$ is positively oriented. 

    Let $S_i=\{\sigma\in \maxx(K_1)| |\sigma|\cap C_i\neq \varnothing\},\, V_i=\{v\in \vertex(K_1)|\exists \sigma\in S_i\colon v\in \sigma\}$.
    Evidently, all $S_i$'s are disjoint. Orientation on $C_i$ induces a cyclic ordering on $S_i$. For $\sigma\in S_i$, denote the next simplex after $\sigma$ with respect to that ordering by $\nexxt(\sigma)$ and  denote the previous simplex by $\previous(\sigma)$.

    Only one edge of any $3$-simplex $\sigma$ in $K$ has both its vertices sent by $f$ to the same vertex of $K_2$. We will call such an edge the \textbf{pivot edge} of $\sigma$. If $\sigma\in S_i$, then its pivot edge is parallel to the segment of $C_i$ that lies in $\sigma$, and the pivot edge inherits the orientation from the orientation of the segment. If the pivot edge $[v,w]$ of $\sigma$ is oriented from $v$ to $w$, we call the vertices the \textbf{initial} and the \textbf{terminal} vertices of $\sigma$, respectively. Therefore, we can define the function $\initial_i\colon S_i \rightarrow V_i$ (resp., $\terminal_i\colon S_i \rightarrow V_i$) which maps $\sigma$ to the initial (resp., terminal) vertex of $\sigma$. The following identity evidently holds for $\sigma\in S_i$:
    \begin{equation}\label{decompose}
        \sigma=(\sigma\cap\nexxt(\sigma))\sqcup \initial_i(\sigma)=(\sigma\cap\previous(\sigma))\sqcup \terminal_i(\sigma). \tag{$\ast$}
    \end{equation}
    
    For $\sigma\in S_i$, let $a_i(\sigma)$ and $b_i(\sigma)$ be the intersection points of $f^{-1}(p)$ with the triangles $\sigma\cap\previous(\sigma)$ and $\sigma\cap\nexxt(\sigma)$ respectively. Then $c_i(\sigma)=[a_i(\sigma), b_i(\sigma)]$.

    Let $p\in \interior(s)$ be the barycenter of $s$ and $q\in \interior(s)$ be a point which does not lie on the medians of $s$. We are ready to formulate and prove the preliminary lemmas.
    
    \begin{lemma}\label{surjectivity}
        If $v\in V_i$ does not lie in the image of $\initial_i$, then it is the common vertex of all simplices of $S_i$. Moreover, there exists a $2$-chain $D_i$ such that $D_i\cap f^{-1}(q)=\varnothing$, $\partial D_i=C_i$.
    \end{lemma}
    \begin{proof}
        As $v\in V_i$, $v\in \sigma$ for some $\sigma\in S_i$. As $\forall \tau\in S_i\colon v\neq in_i(\tau)$, by \eqref{decompose} we have $\forall \tau\in S_i\colon (v\in \tau \Longrightarrow v\in \nexxt(\tau))$. Therefore $\forall \sigma\in S_i\colon v\in V_i$. 

        If $v=\terminal_i(\sigma)$ for some $\sigma\in S_i$, then by \eqref{decompose} we have $v\notin \previous(\sigma)$, which contradicts the previous paragraph. As $v$ does not lie in the images of $in_i$ and $ter_i$, it does not lie in the pivot edge of any $\sigma$. Therefore, $v$ is the only vertex of $V_i$ sent by $f$ to $f(v)$.

        For every $\sigma\in S_i$ consider the triangle $d_i(\sigma)=\conv(c_i(\sigma), v)$. Orient it so that the orientation on $c_i(\sigma)$ coincides with the orientation induced from $d_i(\sigma)$. Consider the $2$-chain $D_i=\sum_{\sigma\in S_i} d_i(\sigma)$. As support of $D_i$ is a subset of $f^{-1}([p, f(v)])$, we have that $D_i\cap f^{-1}(q)=\varnothing$. We also have
    \begin{equation}\label{di}
        \partial D_i=\sum_{\sigma\in S_i} \partial d_i(\sigma) = \sum_{\sigma\in S_i} [v, a_i(\sigma)] + c_i(\sigma) + [b_i(\sigma), v].\tag{$\ast\ast$}
    \end{equation}

    As $[v, a_i(\sigma)]=[v, b_i(\previous(\sigma))]=-[b_i(\previous(\sigma)), v]$ and $[b_i(\sigma), v]=[a_i(\nexxt(\sigma)), v]=$ \\ $=-[v, a_i(\nexxt(\sigma))]$, the first and the third terms of the expression under the sum sign in \eqref{di} cancel out when summed over $\sigma\in S_i$. Then we have $\partial D_i=\sum_{\sigma\in S_i}  c_i(\sigma) = C_i$.
    \end{proof}

    The proof of the next lemma is similar to the proof of the previous one, but we nevertheless describe it in detail for clarity.
    
    \begin{lemma} \label{3repr}
        Let $\alpha$ be a vertex of $s$, i.e. $\alpha=A,B,C$. If no more than two vertices of $V_i$ are sent by $f$ to $\alpha$, then there exists a $2$-chain $D_i$ such that $D_i\cap f^{-1}(q)=\varnothing$, $\partial D_i=C_i$.
    \end{lemma}
    
    \begin{proof}
        Any $\sigma\in S_i$ is sent by $f$ to $s$ surjectively, so for some $v\in V_i$ $f(v)=\alpha$. WLOG the conditions of the previous lemma do not hold, so $v=\initial_i(\sigma)$ for some $\sigma\in S_i$. Then for $u=\terminal_i(\sigma)\neq v$ we also have $f(u)=\alpha$. Thus we have exactly two vertices $v$ and $u$ in $V_i$ which are sent by $f$ to $\alpha$. Subdivide $S_i$ into four subsets:
        \begin{align}
            S_i &=U\sqcup V \sqcup R_u \sqcup R_v, \notag\\
            U &=\{\sigma \in S_i| u\in \sigma, \, v \notin \sigma\}, \notag\\
            V &=\{\sigma \in S_i| v\in \sigma, \, u \notin \sigma\}, \notag\\
            R_u &=\{\sigma \in S_i| \{v,u\}\subset \sigma, \, \initial_i(\sigma)=u\}, \notag\\
            R_v &=\{\sigma \in S_i| \{v,u\}\subset \sigma, \, \initial_i(\sigma)=v\}. \notag
        \end{align}
    
        Notice that
        \begin{align}
            \sigma \in U &\Longrightarrow \nexxt(\sigma) \in U\sqcup R_u, \notag\\
            \sigma \in V &\Longrightarrow \nexxt(\sigma) \in V\sqcup R_v, \notag\\
            \sigma \in R_u &\Longrightarrow \nexxt(\sigma) \in V, \notag\\
            \sigma \in R_v &\Longrightarrow \nexxt(\sigma) \in U. \notag
        \end{align}

        It follows that $|R_u|=|R_v|$. For each $\sigma\in S_i$ consider $d_i(\sigma)=\conv(c_i(\sigma), f^{-1}(\alpha))$. If $\sigma\in U\sqcup V$ then $d_i(\sigma)$ is a triangle, otherwise it is a trapezoid, which we arbitrarily subdivide into two triangles. We orient $d_i(\sigma)$ so that the orientation on $c_i(\sigma)$ coincides with the orientation induced from $d_i(\sigma)$.

        Take $D_i=\sum_{\sigma\in S_i} d_i(\sigma)$. As support of $D_i$ is a subset of $f^{-1}([p, \alpha])$\footnote{The support of $D_i$ is $f^{-1}([p,\alpha])$ intersected with the union of simplices $|\sigma|$ of $S_i$. One could call it the \emph{regular preimage}.}, we have that $D_i\cap f^{-1}(q)=\varnothing$. We also have        
        \begin{align} \label{dii}  \tag{$\ast\ast\ast$}
            \partial D_i &= \sum_{\sigma\in S_i} \partial d_i(\sigma) \notag\\ &= 
            \sum_{\sigma\in U} \left([u, a_i(\sigma)] + c_i(\sigma) + [b_i(\sigma), u]\right) + \sum_{\sigma\in V} \left([v, a_i(\sigma)] + c_i(\sigma) +  [b_i(\sigma), v]\right) \notag\\ &+
            \sum_{\sigma\in R_u} \left([u, a_i(\sigma)] + c_i(\sigma) + [b_i(\sigma), v] + [v,u]\right)  + \sum_{\sigma\in R_v} \left([v, a_i(\sigma)] + c_i(\sigma) + [b_i(\sigma), u] + [u,v]\right). \notag
        \end{align}

    As $[w, a_i(\sigma)]=[w, b_i(\previous(\sigma))]=-[b_i(\previous(\sigma)), w]$ and $[b_i(\sigma), w]=[a_i(\nexxt(\sigma)), w]=$ \\ $=-[w, a_i(\nexxt(\sigma))]$ ($w\in\{u,v\}$), the first and the third terms of the expressions under the sum signs in \eqref{dii} cancel out. As $|R_u|=|R_v|$, the fourth terms of the expressions in the last two sums also cancel out. So $\partial D_i = \sum_{\sigma\in S_i} c_i(\sigma)=C_i$.
\end{proof}
% \begin{remark}
%     The support of $D_i$ constructed in the proofs of the previous lemmas is $f^{-1}([p,\alpha])$ intersected with the union of simplices $|\sigma|$ of $S_i$. One could call it the \emph{regular preimage}.
% \end{remark}

  Let us formulate and prove the main theorem.

  \begin{theorem}\label{mainthm}
    Let $K_1, K_2$ be simplicial complexes such that $|K_1|\cong M$ for an oriented closed $3$-manifold $M$ and $|K_2|\cong S^2$. Let $s$ be a triangle of $K_2$ and $f\colon K_1\rightarrow K_2$ be a simplicial map. If $H(f)\neq 0$, then $\mu(f,s)\geq 9$.
  \end{theorem}

  \begin{proof}
     Assume that all $C_i$'s satisfy the conditions of either lemma~\ref{surjectivity} or lemma~\ref{3repr}. Then for any $i$ there is a $2$-chain $D_i$ such that $\partial D_i=C_i$, $D_i\cap f^{-1}(q)=\varnothing$. Consider $D=\sum_i D_i$, then $H(f)=\lk (f^{-1}(p), f^{-1}(q))=\gamma(D, f^{-1}(q))=0$.

     So for some $C_i$ neither lemmas apply. For such $i$ we have $|S_i|\geq |\Image(\initial_i)|=|V_i|\geq 3\cdot 3=9$. 
  \end{proof}

\subsection{Exactness of the bound}
%To obtain for each $d\in\Z\setminus\{0\}$ a simplicial map $f:S^3\rightarrow S^2$,
Recall that \emph{Seifert fiber spaces} are a particular type of (not necessarily locally trivial) fiber spaces with the total space a 3-manifold and the base a 2-manifold, such that the local triviality condition is violated in a finite number of \textit{exceptional fibers}. $S^3$ can be realized as a Seifert fiber space over $S^2$ in many different ways. For definitions, context and the full classification of Seifert fiber space structures on $S^3$, see \cite{seifert_topologie_1933}, \cite{birman_seifert_1980}, \cite{fomenko_algorithmic_1997}. 

In particular, represent $S^3$ as a union of two solid tori $M$ and $N$ glued by their common boundary $\partial M=\partial N= T$ so that their meridians intersect on the boundary torus $T$ by a single point. Foliate $T$ by $(n,1)$-torus knots ($n\geq 1$). This fibering can be naturally continued into the interiors of the solid tori, so that one of the tori (say, $M$) is fibered locally trivially, the other has (no more than) one exceptional fiber, and the fiber space is $S^2$. The second solid torus is fibered locally trivially if and only if $k=1$, and in this case the projection map is, in fact, the Hopf fibration. All Seifert fiber space structures on $S^3$ with (no more than) one exceptional fiber arise this way.

As we already noted, an example of a simplicial map $\eta\colon S^3_9\rightarrow S^2_4$ with $H(f)=1$ and $\mu(f,s)=9$ was already constructed in~\cite{madahar_sarkaria_simplicial_2000}, and the map is a triangulation of the Hopf fibration.

The maps $\xi=\xi_n\colon S^3_{6n}\rightarrow S^2_4=ABCD, \, n\geq 2$ constructed by Madahar in \cite{madahar_simplicial_2002} are the triangulations of the projection maps of the Seifert fiberings considered above. In all the maps $\xi_n$, the single exceptional fiber is the preimage of the vertex $D$, and the solid torus $M$, which is fibered locally trivially, is the preimage of the triangle $ABC$.  To construct the maps $\zeta_n\colon  S^3\rightarrow S^2$ with $H(\zeta_n)=n, \, \mu(\zeta_n, s)=9$, we will subdivide $S^3_{6n}$ and further define the simplicial map on the newly introduced vertices.

The solid torus $M$ of $S^3_{6n}$ consists of $2n-1$ triangular prisms cyclically glued by their base triangles. We only change the triangulation of $M$ by introducing new vertices in the interior of $S^3_{6n}$, so that the triangulation on the boundary $T=\partial M$ does not change. We introduce $2n-4$ nested triangulated solid tori $M=M_0\supset M_1\supset \dots \supset M_{2n-4}$, so that $M_i$ consists of $2n-1-i$ triangular prisms glued together cyclically\footnote{We cannot introduce yet another torus in a similar pattern because it would be composed of two prisms with the same sets of vertices, which would contradict the condition that $S^3$ is triangulated as a simplicial complex.}. Each vertex $v$ of $M_i$ lies on the segment between a vertex $u$ of $M$ and the barycenter of the only base triangle containing $u$. Define $\zeta_n(v)=\xi_n(u)$. One can also triangulate the spaces between the consecutive nested tori without introducing any more vertices.

%The triangulation of spaces between the consecutively nested tori is further illustrated.

It is easy to see that $\mu(\zeta_n, ABC)=9$ and $[\zeta_n]=[\xi_n]\in \pi_3(S^2)$, so $H(\zeta_n)=H(\xi_n)=n$.

%\section{Final remarks}

\subsection*{Acknowledgements}
   The work of Sergei Vad. Fomin on the lower bound theorem was supported by the Ministry of Science and Higher Education of the Russian Federation (agreement 075-15-2025-344 dated 29/04/2025 for Saint Petersburg Leonhard Euler International Mathematical Institute at PDMI RAS). The work of Mikhail Bludov on the upper bound example is supported by the "Priority 2030" strategic academic leadership program. The authors thank Timur Shamazov for fruitful discussions. The authors thank the Summer Research Programme at MIPT -- LIPS-25 for the opportunity to work on this and other problems.

  \bibliographystyle{alphaurl}
   \bibliography{reference}

\end{document}